\documentclass[secnum,leqno,12pt]{amsart}
\usepackage{amssymb, amsmath,longtable}
\usepackage[all]{xy}
\usepackage[dvipdfmx]{hyperref}
\newtheorem{thm}{Theorem}[section]

\newtheorem{lmm}[thm]{Lemma}

\newtheorem{prp}[thm]{Proposition}

\theoremstyle{definition}
\newtheorem{dfn}[thm]{Definition}
\newtheorem{exa}[thm]{Example}
\newtheorem*{ack}{Acknowledgments}
\newtheorem{prob}[thm]{Problem}
\theoremstyle{remark}
\newtheorem*{rem}{Remark}
\def\Aut{\mathop{\mathrm{Aut}}\nolimits}
\def\im{\mathop{\mathrm{Im}}\nolimits}
\def\kr{\mathop{\mathrm{Ker}}\nolimits}
\def\rk{\mathop{\mathrm{rank}}\nolimits}
\def\Hom{\mathop{\mathrm{Hom}}\nolimits}
\title{Automorphisms of $K3$ surfaces and their applications}
\author[S.~Taki]{Shingo Taki}
\address{Department of Mathematics, Tokai University,
4-1-1, Kitakaname, Hiratsuka, Kanagawa, 259-1292, JAPAN}
\email{taki@tsc.u-tokai.ac.jp}
\urladdr{http://www.sm.u-tokai.ac.jp/~taki/}
\date{\today}
\begin{document}

\begin{abstract} 
This paper is a survey about $K3$ surfaces with an automorphism and log rational surfaces,
in particular, log del Pezzo surfaces and log Enriques surfaces.
It is also a reproduction on my talk at
"Mathematical structures of integrable systems and their applications"
held at Research Institute for Mathematical Sciences in September 2018.
\end{abstract}

\maketitle

\tableofcontents

\section{Introduction}\label{Introduction}
We can see a very interesting mathematical model such that algebra, geometry and analysis 
are harmony through elliptic curves. A $K3$ surface is a 2-dimensional analogue of an elliptic curve.
In algebraic geometry, it is a fundamental problem to study automorphisms of algebraic varieties.
We consider the problem for $K3$ surfaces which are 
the most important and attractive (at least for this author) of complex surfaces.

It is known that the second cohomology of a $K3$ surface has a lattice structure.
Thus the study of $K3$ surfaces can often be attributed to the study 
of lattices by the Torelli theorem \cite{PSS}.
In particular, this viewpoint is effective for the study of automorphisms of a $K3$ surface.

By the definition of $K3$ surfaces, these have a nowhere vanishing holomorphic 2-form. 
A finite group which acts on a $K3$ surface as an automorphism is called 
\textit{symplectic} or \textit{non-symplectic} 
if it acts trivially or non-trivially on a nowhere vanishing holomorphic 2-form, respectively.

For symplectic automorphism groups, we can see a relationship with
Mathieu groups which are sporadic simple groups.
Any finite group of sympletic automorphisms of a $K3$ surface is a subgroup of the 
Mathieu group $M_{23}$ with at least five orbits in its natural action on 24 letters \cite{mukai}.
On the other hand, the study of $K3$ surfaces with non-symplectic symmetry 
has arisen as an application of the Torelli theorem, 
and by now it has been recognized as closely related to classical geometry and special arithmetic quotients.
Most non-symplectic automorphisms give rational surfaces with 
at worst quotient singularities. In particular, some of these 
correspond to log del Pezzo surfaces or log Enriques surfaces.

This article is devoted to studies of non-symplectic automorphisms on $K3$ surfaces, and log rational surfaces.
We see basic results and recent progress for non-symplectic automorphisms.
We summarize the contents of this paper.
In section \ref{what}, we overview algebraic curves and algebraic surfaces via birational viewpoint.
Then we check the position of a $K3$ surfaces in the algebraic surfaces, and 
see basic properties of $K3$ surfaces.
Section \ref{auto} is the main part. We treat automorphisms on $K3$ surfaces. 
Especially, we study the classification of non-symplectic automorphisms in terms of $p$-elementary lattices.
In section \ref{log}, we give applications of theories of non-symplectic automorphisms on $K3$ surfaces.
In particular we apply them to log del Pezzo surfaces and log Enriques surfaces.

We will work over $\mathbb{C}$, the field of complex numbers, throughout this paper.

\begin{ack}
The author would like to thank Professor Shinsuke Iwao who is the organizer of the conference.
He was partially supported by Grant-in-Aid for Young Scientists (B) 15K17520 from JSPS.
This work was supported by the Research Institute for Mathematical Sciences, a Joint Usage/Research Center located in Kyoto University.
\end{ack}

\section{What is a $K3$ surface?}\label{what}
From the beginning of algebraic geometry, it has been understood that 
birationally equivalent varieties have many properties in common.
Thus it is natural to attempt to find in each birational equivalence 
class a variety which is simplest in some sense, and then study these varieties in detail.
In this section,  we check the position of $K3$ surfaces in algebraic surfaces through the birational viewpoint 
and review basic results of $K3$ surfaces.

The following is one of the most important birational invariants.
\begin{dfn}
Let $V$ be a smooth projective variety, $K_{V}$ a canonical divisor of $V$
and $\Phi_{nK_{V}}$ the rational map from $V$ to the projective space
associated with the linear system $|nK_{V}|$. 
For $n \geq 1$, we define the \textit{Kodaira dimension} $\kappa (V)$ to be the 
largest dimension of the image of $\Phi_{nK_{V}}$, 
or $\kappa (V)=-\infty$ if $|nK_{V}| =\emptyset$, hence 
\begin{equation*}
\kappa (V) := 
\begin{cases}
-\infty & |nK_{V}| =\emptyset \\
\max \{\dim \im\Phi_{nK_{V}}(V) | n\geq 1\} & |nK_{V}|\neq \emptyset
\end{cases}.
\end{equation*}
\end{dfn}

First, we recall the theory of  algebraic curves.
But the birational geometry dose not arise for curves because 
a rational map form one non-singular curve to another is in fact morphism.

\subsection{Algebraic curves}
We recall some general results about automorphisms of non-singular algebraic curves, i.e., compact Riemann surfaces.
See \cite[Chapter 4]{gtm52} or \cite[Chapter 2]{gh-pag} for more details.

The most important invariant of an algebraic curve $C$ is its genus $g:=\dim H^{1}(C, \mathcal{O}_{C})$.

\begin{exa}
For an algebraic curve $C$, $C\simeq \mathbb{P}^{1}$ if and only if $g=0$.
\end{exa}

\begin{exa}
We say an algebraic curve is \textit{elliptic} if $g=1$.
The following conditions are equivalent to each other.
\begin{itemize}
\item The genus of $C$ is 1.
\item $C$ is of the form $\mathbb{C}/\Lambda $ for some lattice $\Lambda \subset \mathbb{C}$. 
\item $C$ is realized as a non-singular cubic curve in $\mathbb{P}^{2}$.
\end{itemize}
\end{exa}

\begin{prp}
Every automorphism of $\mathbb{P}^{1}$ is of the form 
\[ \varphi (x) =\frac{ax+b}{cx+d}, \]
where $a,b,c,d \in \mathbb{C}$, $ad-bc\neq 0$.
That is $\Aut (\mathbb{P}^{1})=PGL_{1}(\mathbb{C})$.
\end{prp}

\begin{prp}
Let $C$ be an algebraic curve of genus 1. 
We fix the point $O \in C$. 
The pair $(C,O)$ associates the group  
$\Aut (C,O):=\{\sigma \in \Aut (C) |\sigma (O)=O\}$.
We identify an element of $C$ as a translation.
Then we have the following exact sequence: 
\[ 0\rightarrow C\rightarrow \Aut (C) \rightarrow \Aut (C,O)\rightarrow 0. \]
Hence $\Aut (C)$ is a semi-product,  $\Aut (C)=C\rtimes \Aut (C,O)$.
Indeed, $\Aut (C,O)$ is a finite group of order 2, 4 or 6.
The order depends on the $j$-invariant of $C$.
\end{prp}

\begin{prp}[Hurwitz]
Let $C$ be an algebraic curve of $g\geq 2$.
Then $\Aut (C)$ is a finite group of order at most $84(g-1)$.
\end{prp}

From these propositions we have the following table.

\begin{center}
\begin{longtable}{|c||c|c|c|}
\hline
$C$ & $\mathbb{P}^{1}$ & elliptic curve & general type \\
\hline
$g(C)$ & 0 & 1 & $\geq 2$ \\
\hline
$\kappa (C)$ & $-\infty$ & 0 & 1 \\
\hline
$\Aut (C)$ & $PGL_{1}(\mathbb{C})$ & $C\rtimes$ finite group & finite group \\
\hline
\caption[Algebraic curves]{Algebraic curves}
\end{longtable}
\end{center}

The table implies that an automorphism group of an algebraic curve is an important invariant.
At least, it seems that it has the same information as the Kodaira dimension.

\subsection{Algebraic surfaces}
We treat the structure of birational transformations of algebraic surfaces.
For the details, see \cite[Chapter 5]{gtm52} or \cite[Chapter 2]{CAS}.

For surfaces we see that the structure of birational maps is very simple.
Birational maps between surfaces can be described by monoidal transformations, 
i.e., blowing up a single point.
Any birational transformation of surfaces can be factored into 
a finite sequence of monoidal transformations and their inverses.

\begin{prp}
Let $\phi :S\dashrightarrow S'$ be a birational map.
Then there is a surface $S''$ and a commutative diagram 
$$\xymatrix{
 & &S'' \ar[dl]_{f} \ar[dr]^{g}&   \\
&S\ar@{-->}[rr]_{\phi }& &S'} 
$$
where the morphism $f$, $g$ are composite of monoidal transformations.
\end{prp}

\begin{exa}
The blowing up of $\mathbb{P}^{1}\times \mathbb{P}^{1}$ at a point 
is isomorphic to two times blowing up of $\mathbb{P}^{2}$ at a point. 
\end{exa}

\begin{dfn}
A surface $S$ is \textit{minimal} if every birational morphism $S\rightarrow S'$ is an isomorphism.
\end{dfn}

We show that every surface can be mapped to a minimal surface by a birational morphism. 
Indeed, if $S$ is not minimal, there is some surface $S'$ and a birational morphism $S\rightarrow S'$.
Every algebraic surface has a minimal model in exactly one of the classes in Table \ref{table-cas}.
This model is unique (up to isomorphisms) except for the surfaces with $\kappa =-\infty $.
And the minimal models for rational surfaces
are $\mathbb{P}^{2}$ and the Hirzebruch surfaces $F_{n}$, $n =0, 2, 3\dots $.


\begin{center}
\begin{longtable}{|c|c|c|c|c|c|}
\hline
$\kappa $ & $p_{g}$ & $P_{2}$ & $P_{12}$ & $q$ & $S$ \\
\hline
 & $0$ & $0$ & $0$ & $\geq 1$ & ruled surface \\
\cline{2-6}
\raisebox{1.5ex}[0cm][0cm]{$-\infty $}
 & $0$ & $0$ & $0$ & $0$ & rational surface \\
\hline
 & $1$ & $1$ & $1$ & $2$ & Abelian surface \\
\cline{2-6}
 & $0$ &    & $1$ & $1$ & hyper elliptic surface \\
\cline{2-6}
\raisebox{1.5ex}[0cm][0cm]{$0$}
 & $1$ & $1$ & $1$ & $0$ & $K3$ surface \\
\cline{2-6}
 & $0$ & $1$ & $1$ & $0$ & Enriques surface \\
\hline
$1$ &    &    & $\geq 1$ &    & elliptic surface \\
\hline
$2$ &    & $\geq 1$ & $\geq 1$ &    & general type \\
\hline
\caption[algebraic surfaces]{Classification of algebraic surfaces}\label{table-cas}
\end{longtable}
\end{center}
Here $q=\dim H^{1}(S,\mathcal{O}_{S})$, $p_{g}=\dim H^{2}(S, \mathcal{O}_{S})$ and 
$P_{n}=\dim H^{0}(S, \mathcal{O}_{X}(nK_{S}))$.

The birational geometry of algebraic surfaces was 
largely worked out by the Italian school of algebraic geometry in the years 1890--1910.
Then they inherently had Table \ref{table-cas}.
On the other hand, they knew when birational automorphism groups 
of algebraic surfaces, except $K3$ surfaces and Enriques surfaces, are finite.
Indeed, the Torelli theorem is necessary for studies of automorphisms of $K3$ surfaces and Enriques surfaces.

\subsection{$K3$ surfaces}
We shall give a review of the theory of $K3$ surfaces.
For details, see \cite[Chapter VIII]{BHPV}, \cite{126} and \cite{Lk3}.
Japanese have to read \cite{Ko-book}.

\begin{dfn}
Let $X$ be a compact complex surface. 
If its canonical line bundle $K_{X}$ is trivial and $H^{1}(X, \mathcal{O}_{X})=0$ then 
$X$ is called a \textit{$K3$ surface}
\footnote{The name $K3$ derives from the initials of three Mathematicians \textit{Kummer}, 
\textit{K\"ahler}, \textit{Kodaira} and also from the name of the mountain \textit{K2} in the Karakorum.}. 
\end{dfn}

Any $K3$ surface is K\"ahler (\cite{siu}), and most of them are not algebraic.
But we assume that all $K3$ surfaces are algebraic in this paper.
Indeed, we are interested in finite non-symplectic automorphisms on $K3$ surfaces. 
If $K3$ surface $X$ has a non-symplectic automorphism  of finite order then $X$ is algebraic.
See also Proposition \ref{non-symplectic-algebraic}.

\begin{rem}
Since $K_{X}$ is trivial, $H^{2,0}(X)=H^{0}(X, \Omega_{X}^{2})$ is 1-dimensional by the Serre duality.
Hence $X$ has a nowhere vanishing holomorphic $2$-form $\omega _{X}$.
\end{rem}

\begin{exa}
Let $X\subset \mathbb{P}^{3}$ be a nonsingular quartic surface. 
Then $K_{X}=0$ by the adjunction formula. 
Since $H^{i}(\mathbb{P}^{n},\mathcal{O}_{\mathbb{P}^{n}}(k)) = 0$ for all $0 < i < n$, $k\in \mathbb{Z}$,  
it follows that $H^{1}(X,\mathcal{O}_{X})=0$ from the exact sequence
\[ 0 \to  \mathcal{O}_{\mathbb{P}^{3}}(-4) \to \mathcal{O}_{\mathbb{P}^{3}} \to \mathcal{O}_{X} \longrightarrow  0. \] 
Thus $X$ is a $K3$ surface. 
\end{exa}

\begin{exa}
Let $T$ be a 2-dimensional complex torus and $\iota :T\to T$ the involution $a\mapsto -a$. 
Then there are 16 nodes in $T/\iota $.
The surfaces $X$ given by the minimal resolution $X\to T/\iota $ 
is a $K3$ surface. We call it the \textit{Kummer surface} of $T$. 

Note that if $T$ is not abelian then $X$ is not algebraic.
\end{exa}

By definition, $\chi (\mathcal{O}_{X})=\sum _{i=0}^{2} = \dim H^{i}(X, \mathcal{O}_{X})=2$ 
and the Noether formula thus yields
\[ 2=\chi (\mathcal{O}_{X}) =\frac{K_{X}^{2}+\chi (X)}{12}.\]
Hence the topological Euler number $\chi (X)=24$.
Moreover we have the following result.

\begin{prp}
Let $X$ be a $K3$ surface. Then 
\[ H_{1}(X,\mathbb{Z})=H^{1}(X,\mathbb{Z})=H_{3}(X,\mathbb{Z})=H^{3}(X,\mathbb{Z})=0.\]
And $H^{2}(X,\mathbb{Z})$ is torsion-free.
\end{prp}

We consider the cup product on $H^{2}(X,\mathbb{Z})$:
\[ \langle \ , \ \rangle :H^{2}(X,\mathbb{Z})\times H^{2}(X,\mathbb{Z})\to \mathbb{Z}.\]
Then the pair $(H^{2}(X,\mathbb{Z}), \langle \ , \ \rangle )$
has a structure of a lattice.
Moreover by Wu's formula, the Poincar\'{e} duality and the Hirzebruch index theorem, 
we see that $(H^{2}(X,\mathbb{Z}), \langle \ , \ \rangle )$ is an even unimodular lattice of $\rk 22$ with signature (3,19).
By the classification of even unimodular indefinite lattices (\cite[Chapter 5, $\S 2$, Theorem 5]{Serre}), 
we have $H^{2}(X,\mathbb{Z})\simeq U^{\oplus 3}\oplus E_{8}^{\oplus 2}$.
Throughout this article we shall denote by $A_{m}$, $D_{n}$, $E_{l}$ 
the negative-definite root lattice of type $A_{m}$, $D_{n}$, $E_{l}$ respectively. 
We denote by $U$ the even indefinite unimodular lattice of rank 2.
For a lattice $L$, $L(m)$ is the lattice whose bilinear form is the one on $L$ multiplied by $m$.

\begin{dfn}
Let $\omega _{X}$  be a nowhere vanishing holomorphic $2$-form on $X$.
Set $S_{X}:=\{x\in H^{2}(X,\mathbb{Z})|\langle x, \omega _{X} \rangle =0\}$ and 
$T_{X}:=S_{X}^{\perp }$ in $H^{2}(X,\mathbb{Z})$.
These are called the \textit{N\'{e}ron-Severi lattice} and the \textit{transcendental lattice}, respectively.
\end{dfn}

\begin{prp}
Let $r$ be the Picard number of $X$, i.e. $r=\rk S_{X}$.
Then we have $r \leq 20$.
And $X$ is projective if and only if the signature of $S_{X}$ is $(1, r-1)$,
i.e., $S_{X}$ is a hyperbolic lattice.
\end{prp}

The most interesting structure associated to a $K3$ surface is its weight-two 
Hodge structure on $H^{2}(X,\mathbb{Z})$ given by the decomposition 
\[ H^{2}(X,\mathbb{C})=H^{2,0}(X)\oplus H^{1,1}(X)\oplus H^{0,2}(X). \]

By the definition of $K3$ surfaces, $h^{2,0}(X)=h^{0,2}(X)$.
Therefore we have the following proposition.
\begin{prp}
Let $X$ be a $K3$ surface. Then 
$h^{2,0}(X)=h^{0,2}(X)=1$, $h^{1,1}(X)=20$.
\end{prp}
Since the Hodge decomposition is orthogonal with respect to 
the cup product, it is in fact completely determined by the complex line 
$H^{2,0}(X)\subset H^{2}(X,\mathbb{C})$.

\begin{dfn}
A \textit{Hodge isometry} is an isomorphism 
from $H^{2}(X,\mathbb{Z})$ to $H^{2}(Y,\mathbb{Z})$ 
which preserves the cup product and maps $H^{2,0}(X)$ to $H^{2,0}(Y)$.
\end{dfn}

The following theorem is the most important theorem for $K3$ surfaces.

\begin{thm}[Global Torelli Theorem]\label{gtfk3}
Two $K3$ surfaces $X$ and $Y$ are isomorphic if and only if 
there exists a Hodge isometry 
$\varphi :H^{2}(X,\mathbb{Z})\to H^{2}(Y,\mathbb{Z})$.
If $\varphi $ maps a K\"ahler class on $X$ to 
a K\"ahler class on $Y$, then there exists a unique isomorphism 
$f:Y\to X$ with $f^{\ast }=\varphi $.
\end{thm}

The \textit{period} of a $K3$ surface $X$ is by definition the
natural weight-two Hodge structure on the lattice $H^{2}(X,\mathbb{Z})$.
Thus, Theorem \ref{gtfk3} asserts that two $K3$ surfaces are isomorphic 
if and only if their periods are isomorphic.
The second assertion of Theorem \ref{gtfk3} allows us to describe the 
automorphism group $\Aut (X)$ as the group of Hodge isometries of 
$H^{2}(X,\mathbb{Z})$ preserving K\"ahler classes.

A non-singular rational curve on $X$ is a $(-2)$-vector in $H^{2}(X,\mathbb{Z})$.
Every $(-2)$-class $\delta \in H^{2}(X,\mathbb{Z})$ defines a reflection
\[ s_{\delta }:H^{2}(X,\mathbb{Z})\rightarrow H^{2}(X,\mathbb{Z}), \ \ \ \ x\mapsto x+\langle x, \delta \rangle \delta .\]

For a lattice $L$, we put $W(L):=\langle \{s_{\delta }\in O(L)|\delta \in L, \delta ^{2}=-2\} \rangle$.
Let $\psi $ be an isometry of $L$.
Since $\psi \circ s_{\delta }\circ \psi ^{-1}=s_{\psi (\delta )}$, 
$W(L)$ is a normal subgroup in $O(L)$.

Theorem \ref{gtfk3} says the natural composite homomorphism 
\[ \Aut (X) \to  O(S_{X})\to O(S_{X})/W(S_{X})\]
has a finite kernel or a finite cokernel (see \cite{PSS}).
Hence $\Aut (X)$ and $O(S_{X})/W(S_{X})$ are isomorphic  up to finite groups.
In particular $\Aut (X)$ is finite if and only if $W(S_{X})$ has finite index in $O(S_{X})$.

The problem of describing algebraic $K3$ surfaces with a finite automorphism group
was reduced to a purely algebraic problem, i.e., describe the hyperbolic lattices 
$S$ for which the factor group $O(S)/W(S)$ is finite.
Indeed Nikulin (\cite{Ni3, Ni4}) has completely classified the N\'{e}ron-Severi lattices $S_{X}$
of algebraic $K3$ surfaces with finite automorphism groups.

\section{Automorphisms of $K3$ surfaces}\label{auto}

In this section, we recall some progress on finite automorphism groups of $K3$ surfaces.
Let $G\subset \Aut (X)$ be a finite subgroup.
By definition of $K3$ surfaces, there exists a unique nowhere vanishing 
holomorphic $2$-form on $X$, up to constant.
Hence for every $g\in G$,  there exist some non-zero scalar $\alpha (g)\in \mathbb{C}^{\times }$
which satisfy $g^{\ast }\omega _{X}=\alpha (g)\omega _{X}$. 
Clearly, $\alpha :G\rightarrow \mathbb{C}^{\times }$ is a group homomorphism.
Since $\alpha (G)$ is a subgroup of $\mathbb{C}^{\times }$,  it is a cyclic group of order $I$.
Then we have the following exact sequence
\begin{equation}\label{basic-es}
1\rightarrow \kr \alpha \rightarrow G \rightarrow \mathbb{Z}/I \mathbb{Z} \rightarrow 1. 
\end{equation}

\begin{exa}
Put $F:=X_{0}^{4}+X_{1}^{4}+X_{2}^{4}+X_{3}^{4}$.
Let $X$ be the $K3$ surface defined by the quartic surface $F=0$ in $\mathbb{P}^{3}$.
Clearly, $G:=\mathfrak{S}_{4}\ltimes (\mathbb{Z}/4\mathbb{Z})^{3}$ acts on $X$ as projective transformations.
Since $\omega_{X}$ is given by the Pincar\'{e} residue of 
\[\frac{d\left(\frac{X_{1}}{X_{0}}\right) \wedge d\left(\frac{X_{2}}{X_{0}}\right) \wedge d\left(\frac{X_{3}}{X_{0}}\right)}{F}, \]
we have $1\to \mathfrak{S}_{4}\ltimes (\mathbb{Z}/4\mathbb{Z})^{2} \rightarrow G \to \mathbb{Z}/4\mathbb{Z} \to 1$.
\end{exa}

\begin{dfn}
Let $g$ be an automorphism of $X$.
If $g^{\ast }\omega _{X}=\omega _{X}$ then $g$ is called a \textit{symplectic automorphism}.

Let $G$ be an automorphism group of $X$.
If every $g\in G$ is symplectic then $G$ is called a 
\textit{symplectic automorphism group}.
\end{dfn}

\begin{lmm}
Let $G$ be a finite symplectic automorphism group of $X$. 
Then $G$ has fixed points.
\end{lmm}
\begin{proof}
We assume that $G$ has no fixed points.
Since the natural map $X\to X/G$ is \'{e}tale, 
there exits a nowhere vanishing holomorphic 2-form on $X/G$.
Moreover $H^{0}(X/G,\Omega _{X/G}^{1})=H^{0}(X,\Omega _{X}^{1})=0$.
Hence $X/G$ is a $K3$ surface.

Now, it follows that $\chi (X)=|G|\cdot \chi (X/G)$.
The Euler number of $K3$ surfaces is 24.
This implies that $|G|=1$.
\end{proof}

\begin{thm}[\cite{Ni2}]\label{Ni-fix}
Let $g$ be a symplectic automorphism of order $n$ on $X$.
Then $n \leq 8$. Moreover, the set of fixed points of $g$ 
has cardinality $8, 6, 4, 4, 2, 3$, or $2$, 
if $n=2, 3, 4, 5, 6, 7$, or $8$, respectively.
\end{thm}

Mukai \cite{mukai} has determined all maximum finite symplectic automorphism groups (11 of them);
see also Kondo \cite{Ko-lat} for a lattice-theoretic proof.

\begin{thm}[\cite{mukai}]\label{Mu-sym}
Suppose that $G$ is a finite group of symplectic automorphisms of $K3$ surface. 
Then $G$ is a subgroup of one of the 11 maximum symplectic automorphism groups of $X$ below:

\begin{table}[htb]
\begin{center}
\begin{tabular}{|c|c|}
\hline
$G$ & order \\
\hline
$PSL_{2}(7)$ & 168 \\
\hline
$\mathfrak{A}_{6}$  & 360 \\
\hline
$\mathfrak{S}_{5}$  & 120 \\
\hline
$M_{20} = (\mathbb{Z}/2 \mathbb{Z})^{4} \rtimes \mathfrak{A}_{5}$  & 960 \\
\hline
$F_{384} = (\mathbb{Z}/2 \mathbb{Z})^{4} \rtimes \mathfrak{S}_{4}$ & 384 \\
\hline
$A_{4, 4} = (\mathfrak{S}_{4} \times \mathfrak{S}_{4}) \cap \mathfrak{A}_{8}$  & 288 \\
\hline
$T_{192} = (Q_8 * Q_8) \rtimes  \mathfrak{S}_{3}$ & 192 \\
\hline
$H_{192} = (\mathbb{Z}/2 \mathbb{Z})^{4} \rtimes D_{12}$ & 192 \\
\hline
$N_{72} = (\mathbb{Z}/3 \mathbb{Z})^{2} \rtimes D_{8}$  & 72 \\
\hline
$M_{9} = (\mathbb{Z}/3 \mathbb{Z})^{2} \rtimes Q_{8}$  & 72 \\
\hline
$T_{48} = Q_{8}\rtimes \mathfrak{S}_3$  & 48 \\
\hline
\end{tabular}
\caption[symplectic automorphism groups]{Symplectic automorphism groups}
\end{center}
\end{table}
\end{thm}

First remarks of non-symplectic cases are the following two Propositions.

\begin{prp}[\cite{Ni2}]\label{non-symplectic-algebraic}
If a $K3$ surface $X$ has a finite non-symplectic automorphism $g$ then $X$ is algebraic.
\end{prp}
\begin{proof}
We consider the quotient surface $X/g$. Since $g$ does not act trivially on 
$H^{2}(X, \mathcal{O}_{X})=H^{0}(X, \Omega_{X}^{2})=\mathbb{C}\langle \omega_{X} \rangle$,
we have $H^{2}(X/g, \mathcal{O}_{X/g})=0$.
By the classification of complex surfaces (see \cite[Chapter VI]{BHPV}), $X/g$ is Enriques or rational.
Since an Enriques surface and a rational surface are algebraic, we can pull back an ample class of $X/g$ to $X$.
\end{proof}

\begin{prp}[\cite{Ni2, Xiao, machida-oguiso}]
Suppose that $\mathbb{Z}/I\mathbb{Z}$ is a non-symplectic automorphism group of $X$.
Then $\Phi (I)\leq 21$ and $I\neq 60$, where $\Phi $ is the Euler function. 

Moreover, for each $I$ satisfying $\Phi (I)\leq 21$ and $I\neq 60$, 
there exists a $K3$ surface $X_{I}$ admitting a cyclic group action 
$\langle g \rangle$ with $\langle g \rangle\simeq \langle \alpha (g) \rangle=\mathbb{Z}/I\mathbb{Z}$.
\end{prp}
\begin{table}[htb]
\begin{center}
\begin{tabular}{|c|cccccccccc|}
\hline
$\Phi(I)$ & 20 & 18 & 16 & 12 & 10 & 8 & 6 & 4 & 2 & 1 \\
\hline
    & 66 & 54 & 60 & 42 & 22 & 30 & 18 & 12 & 6 & 2 \\
    & 50 & 38 & 48 & 36 & 11 & 24 & 14 & 10 & 4 & 1 \\
$I$ & 44 & 27 & 40 & 28 &    & 20 &  9 &  8 & 3 & \\
    & 33 & 19 & 34 & 26 &    & 16 &  7 &  5 & & \\
    & 25 &    & 32 & 21 &    & 15 & & & & \\
    &    &    & 17 & 13 &    & & & & & \\
\hline
\end{tabular}
\caption[1]{$\Phi (I) \leq 21$}\label{listI}
\end{center}
\end{table}

The generator of $\mathbb{Z}/I\mathbb{Z}$ is a non-symplectic automorphism of  order $I$.
We call it a \textit{purely} non-symplectic automorphism, hence it satisfies
$g^{\ast}\omega_{X}=\zeta_{I}\omega_{X}$ where $\zeta_{I}$ is a primitive $I$-th root of unity.

\begin{exa}[{\cite[Example 3.2]{keum}}]
We consider the pair of the $K3$ surface and the  automorphism given by the following:
\[ X:y^{2}=x^{3}+t^{11}-t, \ 
g(x,y,t) = (\zeta _{60}^{2}x, \zeta _{60}^{3}y, \zeta _{60}^{6}t). \]
$g$ is non-symplectic and not purely.
Indeed it satisfies $g^{\ast}\omega_{X}=\zeta_{12}\omega_{X}$,
hence $g^{5}$ is a purely non-symplectic automorphism of order 12 and 
$g^{12}$ is a symplectic automorphism of order 5.
\end{exa}

In the exact sequence of (\ref{basic-es}), we have $|\kr \alpha|\leq |M_{20}|=960$ and $I \leq 66$.
The finite automorphism group of a $K3$ surface with largest order is determined by Kondo.

\begin{thm}[\cite{Ko-max}]
Let $G$ be a finite automorphism group of $X$.
Then $|G|\leq 3840$. 
If $|G|=3840$, then $G$ is isomorphic to an extension of $M_{20}$ by $\mathbb{Z}/4 \mathbb{Z}$.
And such pair $(X,G)$ is unique up to isomorphism.
Indeed, $X$ is a Kummer surface.
\end{thm}

Structure of finite non-symplectic automorphism groups is clear.
But a generator (a non-symplectic automorphism) of such a  group is not so.
Non-symplectic automorphisms have been studied by Nikulin who is a pioneer and several mathematicians.
In the following we treat purely non-symplectic automorphisms.

\subsection{Classification of non-symplectic automorphisms}

In this section, we collect some basic results 
for non-symplectic automorphisms on a $K3$ surface. 
For the details, see \cite{Ni2, Ni3, AST}, and so on.


\begin{lmm}\label{sayou}
Let $\sigma$ be a non-symplectic automorphism of order $I$ on a $K3$ surface $X$.
Then the followings are hold.
\begin{itemize}
\item[(1)] The eigen values of $\sigma^{\ast }\mid T_{X}$ are the 
primitive $I$-th roots of unity, hence
$\sigma^{\ast }\mid T_{X}\otimes \mathbb{C}$ can be diagonalized as:
\[ \begin{pmatrix} 
\zeta_{I} E_{q} & 0 & \cdots & \cdots & \cdots & 0 \\ 
\vdots &  & \ddots &  &  & \vdots \\ 
\vdots &  &  & \zeta_{I}^{n} E_{q} &  & \vdots \\ 
\vdots &  &  &  & \ddots  & 0 \\ 
0 & \cdots & \cdots & \cdots & 0 & \zeta_{I}^{I-1} E_{q} \\ 
\end{pmatrix}, \]
where $E_{q}$ is the identity matrix of size $q$ 
and $1\leq n\leq I-1$ is co-prime with $I$.

\item[(2)] Let $P^{i,j}$ be an isolated fixed point of $\sigma$ on $X$. 
Then $\sigma^{\ast }$ can be written as 
\[ \begin{pmatrix}  \zeta_{I} ^{i} & 0 \\ 0 & \zeta_{I} ^{j}  \end{pmatrix}  \hspace{5mm} (i+j \equiv 1 \mod I) \]
under some appropriate local coordinates around $P^{i,j}$.
\item[(3)] Let $C$ be an irreducible curve in $X^{\sigma}$ and $Q$ a point on $C$. 
Then $\sigma^{\ast }$ can be written as
\[ \begin{pmatrix}  1 & 0 \\ 0 & \zeta_{I}   \end{pmatrix} \] 
under some appropriate local coordinates around $Q$. 
In particular, fixed curves are non-singular.
\end{itemize}\end{lmm}

Lemma \ref{sayou} (1) implies that $\Phi (I)$ divides $\rk T_{X}$, where $\Phi$ is the Euler function.
Lemma \ref{sayou} (2) and (3) imply that 
the fixed locus of $\sigma$ is either empty or the disjoint union of non-singular curves $C_{l}$ and isolated points $P_{k}^{i_{k},j_{k}}$:
\[ X^{\sigma}=\{ P_{1}^{i_{1}, j_{1}}, \dots , P_{M}^{i_{M}, j_{M}} \} \amalg C_{1} \amalg \dots \amalg C_{N}. \]

The global Torelli Theorem gives the following.
\begin{rem}\label{Torelli}
Let $X$ be a $K3$ surface and $g_{i}$ ($i=1$, $2$) automorphisms of $X$
such that $g_{1}^{\ast}|S_{X}=g_{2}^{\ast}|S_{X}$ and that
$g_{1}^{\ast}\omega _{X}=g_{2}^{\ast}\omega _{X}$.
Then $g_{1}=g_{2}$ in $\Aut (X)$.
\end{rem}
The Remark says that for study of non-symplectic automorphisms, 
the action on $S_{X}$ is important. Hence the invariant lattice 
$S_{X}^{\sigma}:=\{x\in S_{X}| \sigma^{\ast} (x) =x\}$ plays an essential role
for the classification of non-symplectic automorphisms.

\begin{thm}[\cite{Ni3}]\label{ord2-f}
Let $\sigma$ be a non-symplectic involution. 
Then $S_{X}^{\sigma}$ is a 2-elementary lattice
\footnote{See also \cite{Ni1} and \cite{R-S} for $p$-elementary lattices.}, 
hence, $\Hom (S_{X}^{\sigma}, \mathbb{Z})/S_{X}^{\sigma}=(\mathbb{Z}/2\mathbb{Z})^{\oplus a}$.
And the fixed locus of $\sigma$ is of the form 
\begin{equation*}
X^{\sigma}=
\begin{cases}
\phi  & \text{$S_{X}^{\sigma}=U(2)\oplus E_{8}(2)$}, \\
C^{(1)}\amalg C^{(1)} & \text{$S_{X}^{\sigma}=U\oplus E_{8}(2)$}, \\
C^{(g)} \amalg \mathbb{P}^{1} \amalg \dots \mathbb{P}^{1} & \text{otherwise},
\end{cases}
\end{equation*}
where $C^{(g)}$ is a genus $g$ curve with $g=(22-\rk S_{X}^{\sigma}-a)/2$.
Moreover the number of $\mathbb{P}^{1}$ is given by $(\rk S_{X}^{\sigma}-a)/2$.
\end{thm}

\begin{exa}
Let $C$ be a smooth sextic curve in $\mathbb{P}^{2}$ and 
$X \to \mathbb{P}^{2}$ the double cover branched along $C$.
Then $X$ is a $K3$ surfaces and the covering transformation induces 
a non-symplectic involution $\iota$.

Since the N\'{e}ron-Severi lattice $S_{X}$ consists of the pull-back of a hyperplane in $\mathbb{P}^{2}$, 
$S_{X}^{\sigma}=S_{X}=A_{1}(=\langle -2 \rangle)$ with $(\rk S_{X}, a)=(1,1)$.
The fixed locus of $\iota$ is a genus 10 curve coming from $C$.
Indeed we have $10=(22-1-1)/2$ and $0=(1-1)/2$.
\end{exa}

As the same as these, fixed loci of non-symplectic automorphisms of order $p^{k}$
characterized in terms of the invariants of $p$-elementary lattices.

\begin{thm}[{\cite{AS, Taki1}}]\label{ord3-f}
Let $\sigma$ be a non-symplectic automorphism of order 3.
Then $S_{X}^{\sigma}$ is a 3-elementary lattice
hence, $\Hom (S_{X}^{\sigma}, \mathbb{Z})/S_{X}^{\sigma}=(\mathbb{Z}/3\mathbb{Z})^{\oplus a}$.
And the fixed locus of $\sigma$ is of the form 
\[ X^{\sigma}=
C^{(g)} \amalg \mathbb{P}^{1} \amalg \dots \mathbb{P}^{1} \amalg\{P_{1},\dots , P_{n}\}
\]
where $C^{(g)}$ is a genus $g$ curve with $g=(22-\rk S_{X}^{\sigma}-2a)/4$ and 
$P_{i}$ are isolated points.
Moreover the number of $\mathbb{P}^{1}$ is given by $(2+\rk S_{X}^{\sigma}-2a)/4$ and
$n=(\rk S_{X}^{\sigma}-2)/2$.
In the case $(\rk S_{X}^{\sigma}, a)= (8,7)$ for which $(g, \sharp \mathbb{P}^{1})=(0, -1)$, 
this means a fixed locus consisting of 3 isolated points and no curve component.
\end{thm}

We do not have the complete classification of non-symplectic automorphisms.
See \cite{AlST, AlS, ord4, AST, 13-19, Schutt, Taki2, Taki3, Taki4} 
for cases of pime-power order
and \cite{Kondo1, machida-oguiso, Xiao, Br} for cases of non-pime-power order.

\begin{prob}
Classify non-symplectic automorphisms of order 4 and 9 under generic conditions.
\end{prob}
It seems that the Problem is difficult when the quotient surface of a $K3$ surface by
a non-symplectic automorphism is a log Enriques surface.

\begin{rem}
Moduli spaces of $K3$ surfaces with a non-symplectic automorphism 
have also been studied. For example, see \cite{ma, mot}.
\end{rem}

\section{Log rational surfaces}\label{log}

Let $Z$ be a normal algebraic surface with at worst log terminal singularities (i.e., quotient singularities).
$Z$ is called a \textit{log del Pezzo} if the anticanonical divisor $-K_Z$ is ample.
$Z$ is called \textit{log Enriques} if the irregularity $\dim H^{1}(Z, \mathcal{O}_{Z})=0$
and a positive multiple $IK_{Z}$ of a canonical Weil divisor $K_{Z}$ is linearly equivalent to zero.
These surfaces constitute one of the most interesting classes of rational surfaces;
they naturally appear in the outputs of the (log) minimal model program and their classification 
is an interesting problem. 
The \textit{index} $I$ of $Z$ is the least positive integer such that $I K_Z$ is a Cartier divisor.

\subsection{log del Pezzo surfaces}
See also \cite{AN, Nakayama, OT} for details.

Log del Pezzo surfaces with index $I=1$ are sometimes called Gorenstein del Pezzo surfaces and
their classification is a classical topic. 
In the index $I=2$ Alexeev and Nikulin \cite{AN} (over $\mathbb{C}$)
and Nakayama \cite{Nakayama} (char. $p\geq 0$ and also for log pairs) 
gave complete classifications, whose methods are independent in nature. 
Ohashi and Taki \cite{OT}  discuss a generalization of 
the ideas of \cite{AN} to treat log del Pezzo surfaces of index three.
We review the classification of log del Pezzo surfaces of index 2.

\begin{thm}[\cite{AN}]
Let $Z$ be a log del Pezzo surface of index $\leq 2$.
The followings hold:
\begin{itemize}
\item[(1)] There exists a branched covering $X\to Z$ such that 
$X$ is a $K3$ surface with a non-symplectic automorphism of order 2.
Moreover the automorphism fixes a non-singular curve with genus $\geq 2$.
\item[(2)] Let $\iota$ be a non-symplectic automorphism of order 2 on a $K3$ surface $X$.
If the fixed locus of $\iota$ contains a non-singular curve with genus $\geq 2$ then
we have a log del Pezzo surface of index 2 by contracting some curves on $X/\iota$.
\item[(3)] We can study log del Pezzo surfaces which are constructed in (2)
by using techniques of $K3$ surfaces.
\end{itemize}
\end{thm}
By the theorem and Theorem \ref{ord2-f}, we can study all log del Pezzo surfaces of index $\leq 2$.
The following is the case of index 3.

\begin{thm}[\cite{OT}]
Let $Z$ be a log del Pezzo surface of index 3.
Assume that the linear system $|-3K_{Z}|$ contains a divisor $2C$ 
where $C$ is a smooth curve which does not intersect the singularities.
(We call the assumption the \textit{Multiple Smooth Divisor Property}.)
The followings hold:
\begin{itemize}
\item[(1)] There exists a branched covering $X\to Z$ such that 
$X$ is a $K3$ surface with a non-symplectic automorphism of order 3.
Moreover the automorphism fixes a non-singular curve with genus $\geq 2$.
\item[(2)] Let $\sigma$ be a non-symplectic automorphism of order 3 on a $K3$ surface $X$.
If the fixed locus of $\sigma$ contains a non-singular curve with genus $\geq 2$ then
we have a log del Pezzo surface of index 3 by contracting some curves on $X/\sigma$.
\item[(3)] We can study log del Pezzo surfaces which are constructed in (2)
by using techniques of $K3$ surfaces.
\end{itemize}
\end{thm}
By the theorem and Theorem \ref{ord3-f}, we can study log del Pezzo surfaces of index 3
which have the multiple smooth divisor property.
But there exists a log del Pezzo surfaces of index 3 which does not correspond to
a $K3$ surfaces with a non-symplectic automorphism of order 3, hence 
there exists a log del Pezzo surface which does not satisfy the multiple smooth divisor property.

\begin{lmm}
Let $Z$ be a log del Pezzo surface of index 3 which has the multiple smooth divisor property.
Then we have $K_{Z}^{2}=8(g(C)-1)/3$.
\end{lmm}
\begin{proof}
Assume that the linear system $|-3K_{Z}|$ contains a divisor $2C$ 
where $C$ is a smooth curve which does not intersect the singularities.
Then we have 
\begin{align*}
2g(C)-2&=C^{2}+C.K_{Z}\\
&=\frac{9}{4}C^{2}-\frac{3}{2}K_{Z}^{2}\\
&=\frac{3}{4}K_{Z}^{2}
\end{align*}
by the genus formula.
\end{proof}

\begin{exa}
The weighted projective space $Z_{1}=\mathbb{P}(1,1,3)$ is a log del Pezzo surface of index 3.
But it does not satisfy the multiple smooth divisor property.
Because $K_{Z_{1}}^{2}=25/3$.

We put $Z_{2}=\mathbb{P}(1,2,9)$.
It is easy to see that $Z_{2}$ is a log del Pezzo surface of index 3 and
has singularities at $(0,0,1)$ and $(0,1,0)$.
Note that $|-3/2K_{Z_{2}}|=|\mathcal{O}_{Z_{2}}(18)|$.
Let $C$ be an element of $\mathcal{O}_{Z_{2}}(18)$ defined by
$\{x^{18}+y^{9}+z^{2}+(\text{terms of degree 18})=0\}$
where $x, y$ and $z$ are homogeneous coordinates of $Z_{2}$.
Then the smooth divisor $C$ does not pass through $(0,0,1)$ and $(0,1,0)$.
Hence $Z_{2}$ satisfies the multiple smooth divisor property.
\end{exa}

\begin{exa}
Let $Z$ be a weighted hypersurface $(10)$ in $\mathbb{P}(1,1,5,9)$.
Note that $Z$ is a log del Pezzo surface with a singular point induced by $(0,0,0,1)$
and $\mathcal{O}(K_{Z})\simeq \mathcal{O}_{Z}(10-1-1-5-9)=\mathcal{O}_{Z}(-6)$.
Let $C$ be an element of $\mathcal{O}_{Z}(9)$ defined by 
$\{x^{9}+y^{9}+x^{2}y^{2}z+w+\dots =0 \}$ where 
$x, y, z$ and $w$ are homogeneous coordinates of $\mathbb{P}(1,1,5,9)$.
Then the smooth divisor $C$ does not pass through $(0, 0, 0, 1)$. 
Hence $Z$ satisfies the multiple smooth divisor property.

Let $\nu:\widetilde{Z}\to Z$ be the minimal resolution.
Then we have $K_{\widetilde{Z}}=\nu^{\ast}K_{Z}-1/3E-2/3F$.
Here $E$ and $F$ are smooth rational curves such that $E^{2}=-2$, $F^{2}=-5$ and $E.F=1$.
By blowing-up at the intersection point of $E$ and $F$,
we obtain $Z_{r}\to Z$ which is called the \textit{right resolution}.

Note that $Z_{r}$ has a ($-3$)-curve, a ($-6$)-curve and the strict transform of $C$.
Let $\pi: \widetilde{X} \to Z_{r}$ be the triple cover branched along these curves.
By contracting of the $(-1)$-curve induced from the ($-3$)-curve: 
$\widetilde{X}\to X$, we have a $K3$ surfaces $X$
and the covering transformation induces a non-symplectic automorphism $\sigma$ 
of order 3 which fixes a genus 4 curve, a smooth rational curve and an isolated point.
We remark that these correspond to $C$, $F$ and $E$, respectively.

By Theorem \ref{ord3-f} and the classification of 3-elementary lattices,
we have $S_{X}=S_{X}^{\sigma}=U\oplus A_{2}$.
\end{exa}

\begin{rem}
Fujita and Yasutake \cite{Fujita-Yasutake}
have given complete classification of log del Pezzo surfaces of index 3.
The technique for the classification based on the argument of \cite{Nakayama}.
\end{rem}

\begin{prob}
Under appropriate assumptions, study log del Pezzo surfaces of index 5 corresponding to 
$K3$ surfaces with an non-symplectic automorphism of order 5.
(But it seems that most log del Pezzo surfaces of index 5 do not 
correspond to $K3$ surfaces with an non-symplectic automorphism.)
\end{prob}

\subsection{log Enriques surfaces}
See \cite{Z1, Z2, Taki5} for details.
 
Without loss of generality, we assume that a log Enriques surface $Z$ has no Du Val singular points,
because if $Z' \to Z$ is the minimal resolution of all Du Val singular points of $Z$ then
$Z'$ is also a log Enriques surface of the same index of $Z$.

Let $Z$ be a log Enriques surface of index $I$.
The Galois $\mathbb{Z}/I\mathbb{Z}$-cover
\[ \pi: Y:= \text{Spec}_{{\mathcal{O}}_{Z}} \left( \bigoplus_{i=0}^{I-1}\mathcal{O}_{Z}(-iK_{Z}) \right) \to Z \]
is called the (global) \textit{canonical covering}.
Note that $Y$ is either an abelian surface or a $K3$ surface with at worst Du Val singular points, 
and that $\pi$ is unramified over $Z\setminus \text{Sing}(Z)$.
A log Enriques surface $Z$ is \textit{of type} $A_{m}$ or $D_{n}$ if, by definition,
its canonical cover $Y$ has a  singular point of type $A_{m}$ or $D_{n}$, respectively.

It is interesting to consider the index $I$ of a log Enriques surface.
Blache \cite{Bl, Z1} proved that $I\leq 21$. Thus if $I$ is prime then 
$I=2, 3, 5, 7, 11, 13, 17$ or 19.

\begin{thm}[{\cite{logEnriques19, logEnriques18, OZorder5, 11, 13-19}}]\label{rekishi}
The followings hold:
\begin{itemize}
\item[(1)] There is one log Enriques surface of type $D_{19}$ (resp. $A_{19}$, $D_{18}$), up to isomorphism.
\item[(2)] There are  two log Enriques surfaces of type $A_{18}$, up to isomorphism.
\item[(3)] There are  two log Enriques surfaces of index 5 and type $A_{17}$, up to isomorphism.
\end{itemize}
The followings do not refer to singular points.
But these determine log Enriques surfaces with large prime indices:
\begin{itemize}
\item[(4)]  There are  two maximal log Enriques surfaces of index 11, up to isomorphism.
\item[(5)] If $I$=13, 17 or 19 then there is a unique log Enriques surface of index $I$, up to isomorphism.
\end{itemize}
\end{thm}
\begin{rem}
If a log Enriques surface is of type $A_{19}$ (resp. $A_{18}$, $D_{18}$ or $D_{19}$) 
then its index is 2 (resp. 3).
\end{rem}

To prove Theorem \ref{rekishi}, we studied non-symplectic automorphisms of $K3$ surfaces, 
because the canonical covering $\pi$ is a cyclic Galois covering of order $I$ which 
acts faithfully on the space $H^{0}(Y,\mathcal{O}_{Y}(K_{Y}))$.
And we have gotten the following.

\begin{thm}[{\cite{logEnriques19, OZorder5, 11, 13-19}}]\label{32511}
Let $\sigma_{I}$ be a non-symplectic automorphism of order $I$ on a $K3$ surface $X_{I}$ and
$X_{I}^{\sigma_{I}}$ be the fixed locus of $\sigma_{I}$;
$X_{I}^{\sigma_{I}}=\{x\in X_{I}|\sigma_{I}(x)=x \}$.
Then the followings hold:
\begin{itemize}
\item[(1)]
If $X_{3}^{\sigma_{3}}$ 
consists of only (smooth) rational curves and possibly some isolated points, 
and contains at least 6 rational curves then 
a pair ($X_{3}$, $\langle \sigma_{3} \rangle$) is unique up to isomorphism.
\item[(2)]
If $X_{2}^{\sigma_{2}}$ 
consists of only (smooth) rational curves 
and contains at least 10 rational curves then 
a pair ($X_{2}$, $\langle \sigma_{2} \rangle$) is unique up to isomorphism.
\item[(3)]
If $X_{5}^{\sigma_{5}}$ 
contains no curves of genus $\geq 2$, but contains at least 3 rational curves
then a pair ($X_{5}$, $\langle \sigma_{5} \rangle$) is unique up to isomorphism.
\item[(4)]
Put $M:=\{x\in H^{2}(X_{11},\mathbb{Z})| \sigma_{11}^{\ast} (x) =x\}$.
A pair ($X_{11}$, $\langle \sigma_{11} \rangle$) is unique up to isomorphism
if and only if $M=U\oplus A_{10}$.
\item[(5)] Pairs ($X_{13}$, $\langle \sigma_{13} \rangle$),  ($X_{17}$, $\langle \sigma_{17} \rangle$) and
($X_{19}$, $\langle \sigma_{19} \rangle$) are unique up to isomorphism, respectively.
\end{itemize}
\end{thm}

These theorems miss the case of $I=7$.
Recently we have the following.
\begin{thm}[\cite{Taki5}]
The followings hold:
\begin{itemize}
\item[(1)] There is, up to isomorphism, only one log Enriques surface of index 7 and type $A_{15}$.
\item[(2)] If $X_{7}^{\sigma_{7}}$ consists of only smooth rational curves and some isolated points
and contains at least 2 rational curves then
a pair ($X_{7}$, $\langle \sigma_{7} \rangle$) is unique up to isomorphism.
\end{itemize}
\end{thm}

\begin{exa}[{\cite[Example 6.1 (3)]{AST}}]\label{exk3w7}
Put 
\[ X_{\text{AST}}:y^{2}=x^{3}+\sqrt[3]{-27/4}x+t^{7}-1, \ 
\sigma_{\text{AST}} (x,y,t) = (x, y, \zeta _{7}t). \]
Then $X_{\text{AST}}$ is a $K3$ surface with $S_{X_{\text{AST}}}=U\oplus E_{8}\oplus A_{6}$
and $\sigma_{\text{AST}}$ is a non-symplectic automorphism of order 7.
Note that $X_{\text{AST}}$ has one singular fiber of type I$_{7}$ over $t=0$, one singular fiber of type II$^{\ast}$ over $t=\infty$
and 7 singular fibers of type I$_{1}$ over $t^{7}=1$.
\end{exa}

\begin{exa}
We consider the pair  ($X_{\text{AST}}$, $\langle \sigma_{\text{AST}} \rangle$) in Example \ref{exk3w7}.
Let $f:X_{\text{AST}}\to Y$ be the contraction of the following rational tree $\Delta _{\text{AST}}$ of Dynkin type $A_{15}$ to a point $Q$:
\[ \Gamma_{2}-\Gamma_{3}-\Gamma_{4}-\Gamma_{5}-\Gamma_{6}-\Gamma_{7}-S
-\Theta_{1}-\Theta_{2}-\Theta_{3}-\Theta_{4}-\Theta_{5}-\Theta_{6}-\Theta_{7}-\Theta_{8}, \]
where $S$ is a cross-section, $\Gamma_{i}$ is a component of a singular fiber of type I$_{7}$
and $\Theta_{j}$ is a component of a singular fiber of type II$^{\ast}$.
Here a singular fiber of type  I$_{7}$ is given by $\sum_{i=1}^{7}\Gamma_{i}$
which $\Gamma_{7}$ meets $S$, 
and a singular fiber of type II$^{\ast}$ is given by $\sum_{j=1}^{6}j\Theta_{j}+4\Theta_{7}+2\Theta_{8}+3\Theta_{9}$.
Hence $\Gamma_{7}$ and $\Theta_{6}$ are fixed curves of $\sigma_{\text{AST}}$.

Then $\sigma_{\text{AST}}$ induces an automorphism $\tau$ on $Y$ so that $Y^{\tau}:=\{Q, f(P) \}$
where $P$ is the isolated fixed point of type $P^{2,6}$ on $\Theta_{9}$.
Now the quotient surface $Z_{\text{AST}}:=Y/\tau$ is a log Enriques surface of index 7 and type $A_{15}$.
Note that $Z_{\text{AST}}$ has exactly two singular points
under the two fixed points $Q$ and $f(P)$.
\end{exa}


\end{document}